\documentclass[10pt]{article}
\usepackage{lmodern}
\usepackage{amsmath}
\usepackage[T1]{fontenc}
\usepackage[utf8]{inputenc}
\usepackage{authblk}
\usepackage{amsfonts}
\usepackage{graphicx}
\usepackage{color,tikz}

\usepackage{rotating}
\usepackage{amssymb}
\usepackage[english]{babel}
\usepackage{color}
\usepackage{amsthm}
\usepackage{graphicx}
\usepackage{mathrsfs}
\usepackage{makecell}
\usepackage{microtype}
\usepackage{enumerate}
\usepackage{mathscinet}
\usepackage{array}
\usepackage{multirow}
\usepackage{booktabs}
\usepackage{enumerate}
\usepackage[cal=boondoxo,bb=ams]{mathalfa}
\usepackage{hyperref}
\hypersetup{hidelinks}
\usepackage{titlesec}
\newtheorem{theorem}{Theorem}[section]
\newtheorem{prop}{Proposition}[section]
\newtheorem{lemma}{Lemma}[section]

\newtheorem{remark}{Remark}[section]

\newcommand{\ml}{\mathcal}
\newcommand{\mb}{\mathbb}

\DeclareMathOperator{\intt}{int}
\DeclareMathOperator{\extt}{ext}
\DeclareMathOperator{\bdd}{bdd}

\title{Large-time asymptotic behaviors for linear Blackstock's model of thermoviscous flow}
\author[1,2]{Wenhui Chen\thanks{Wenhui Chen (wenhui.chen.math@gmail.com)}}
\affil[1]{School of Mathematics and Information Science, Guangzhou University, 510006 Guangzhou, China}
\affil[2]{School of Mathematical Sciences, Shanghai Jiao Tong University, 200240 Shanghai, China}
\author[3]{Hiroshi Takeda\thanks{Hiroshi Takeda (h-takeda@fit.ac.jp)}}
\affil[3]{Department of Intelligent Mechanical Engineering, Faculty of Engineering, Fukuoka Institute of Technology,  811-0295 Fukuoka, Japan}

\date{}
\begin{document}

\maketitle
\begin{abstract}
	\medskip
In the classical theory of acoustic waves, Blackstock's model was proposed in 1963 to characterize the propagation of sound in thermoviscous fluids. In this paper, we investigate large-time asymptotic behaviors of the linear Cauchy problem for general Blackstock's model (that is, without Becker's assumption on monatomic perfect gases). We derive first- and second-order asymptotic profiles of solution as $t\gg1$ by applying refined WKB analysis and Fourier analysis. Our results not only improve optimal estimates in [Chen-Ikehata-Palmieri, \emph{Indiana Univ. Math. J.} (2023)] for lower dimensional cases, but also illustrate the optimal leading term and novel second-order profiles of solution with additional weighted $L^1$ data. \\
	
	\noindent\textbf{Keywords:} acoustic waves, third-order evolution equation, Cauchy problem, optimal estimates, asymptotic profiles, optimal leading term.\\
	
	\noindent\textbf{AMS Classification (2020)} 35B40, 35G10, 35C20.
\end{abstract}
\fontsize{12}{15}
\selectfont
\section{Introduction}
In the last two decades, theoretical studies of acoustic waves are generally used in medical and industrial applications of high-intensity ultra sound, including lithotripsy, thermotherapy and sonochemistry. In order to characterize the propagation of sound in thermoviscous fluids, several mathematical models were proposed recently (e.g. \emph{Westervelt's equation}, \emph{Kuznetsov's equation} and \emph{Jordan-Moore-Gibson-Thompson equation}), which can be understood as higher-order wave equations with viscoelastic damping (see \cite{Hamilton-Blackstock-1998,Kaltenbacher-Thalhammer-2018} and references therein). Specially among them, in 1963, one of the fundamental models in acoustic waves that is the so-called \emph{Blackstock's model}, has been established in the pioneering work \cite{Blackstock-1963} by David T. Blackstock (Eu\'gene P. Schoch Professor Emeritus in The University of Texas at Austin). The detailed  elaboration of Blackstock's model with its applications has been given in the monograph \cite{Hamilton-Blackstock-1998}. In this paper, we will study some qualitative properties of solution to the linear Blackstock's model in $\mb{R}^n$ for large-time. Before introducing our main purposes, we will sketch out some historical background of Blackstock's model from the physic and mathematics points of view.
\subsection{Background of Blackstock's model}\label{Sub-1.1}
\textbf{Physical background of Blackstock's model:} Concerning sound propagation in thermoviscous fluids, the classical theory of acoustic waves considers the conservation of mass, the conservation momentum (i.e. the Navier-Stokes equations), the conservation energy associated with Fourier's law of heat conduction to the heat flux vector, and the state equation for perfect gases, which arises the well-known \emph{Navier-Stokes-Fourier} (N-S-F) equations. Under the irrotational flow, by applying Lighthill scheme of approximations (i.e. reserving the first- and second-order components only with small perturbations near the equilibrium state) to the N-S-F equations, the classical Blackstock's model (see \cite[Equations (4) or (6)]{Brunnhuber-Jordan-2016})
\begin{align}\label{Classical-B-01}
\left(\partial_t-\frac{\bar{\nu}}{\mathrm{Pr}}\Delta\right)\left(\bar{\psi}_{tt}-c_0^2\Delta\bar{\psi}-\bar{\delta}\Delta\bar{\psi}_t\right)-\frac{\bar{\nu}}{\mathrm{Pr}}(\bar{\delta}-b\gamma \bar{\nu})\Delta^2\bar{\psi}_t=\partial_t\left(\partial_t|\nabla\bar{\psi}|^2+(\gamma-1)\bar{\psi}_t\Delta\bar{\psi}\right),
\end{align}
has be established by David T. Blackstock \cite{Blackstock-1963} in 1963. Concerning more detailed deductions of the model \eqref{Classical-B-01} from the N-S-F equations, we refer interested readers to \cite[Section 1]{Brunnhuber-2015} or \cite[Appendix A]{Kaltenbacher-Thalhammer-2018}. In the above model, the scalar unknown $\bar{\psi}=\bar{\psi}(t,x)$ stands for the acoustic velocity potential due to the irrotational assumption. To explain the physical quantities \eqref{Classical-B-01} clearly, we collect them into Table \ref{Table_1} as follows:
\begin{table}[http]
	\centering	
	\caption{Notations for physical quantities and auxiliary abbreviations}
	\begin{tabular}{llll}
		\toprule
		Quantity & Notation   & Quantity & Notation   \\
		\midrule
		Bulk viscosity  & $\mu_{\mathrm{B}}>0$&Thermal diffusivity & $\bar{\kappa}>0$\\
		Shear viscosity  & $\mu_{\mathrm{V}}>0$&Prandtl number & $\mathrm{Pr}=\bar{\nu}/\bar{\kappa}$ \\
		 Speed of sound& $c_0>0$&Ratio of specific heats & $\gamma\in(1,5/3]$\\
	 Viscosity number & $b=4/3+\mu_{\mathrm{B}}/\mu_{\mathrm{V}}$&Diffusivity of sound & $\bar{\delta}=\bar{\nu}b+(\gamma-1)\bar{\nu}/\mathrm{Pr}$\\
	Kinematic viscosity& $\bar{\nu}>0$& & \\
		\bottomrule
	\end{tabular}
	\label{Table_1}
\end{table}

\noindent\textbf{Simplification of the model by rescaling:} For the sake of transparency, we can rescale the scalar equation \eqref{Classical-B-01} by a modified acoustic velocity potential $\psi(t,x):=\bar{\psi}(t,c_0x)$ so that the equation concisely turns into
\begin{align}\label{Non-linear_Blackstock_2}
	\underbrace{\left(\partial_t-\frac{\nu}{\mathrm{Pr}}\Delta\right)(\psi_{tt}-\Delta\psi-\delta\Delta\psi_t)-\frac{\nu}{\mathrm{Pr}}(\delta-b\gamma\nu)\Delta^2\psi_t}_{\mbox{linear Blackstock's model}}=\frac{1}{c_0^2}\partial_{t}\left(\partial_t|\nabla\psi|^2+(\gamma-1)\psi_t\Delta\psi\right),
\end{align}
where we took the modified physical quantities $(\nu,\delta,\kappa):=(\bar{\nu},\bar{\delta},\bar{\kappa})/c_0^2$.  In reality, the modified thermal diffusivity $\kappa=\bar{\kappa}/c_0^2$ is always a small number in perfect gases, for example, oxygen  with $\bar{\kappa}=23.2$ mm/s and $c_0=326$ m/s; xenon with $\bar{\kappa}=6.27$ mm/s and $c_0=178$ m/s; nitrogen with $\bar{\kappa}=19.6$ mm/s and $c_0=349$ m/s. Without loss of generality, we consequently assume $0<\kappa\ll \delta$ throughout this paper.
\medskip

\noindent\textbf{Becker's assumption:} Concerning the physical quantities in the linear part of Blackstock's model \eqref{Non-linear_Blackstock_2}, most of the recent works studied a particular situation for Blackstock's model under the so-called \emph{Becker's assumption} (see \cite{Becker-1922,Morduchow-Libby-1949,Hayes-1960} and references therein), which considers the fluid to be a monatomic perfect gas. In the mathematical language, Becker's assumption means the vanishing bulk viscosity $\mu_{\mathrm{B}}=0$ as well as
\begin{align}\label{Becker_Assum}
	b+\frac{\gamma-1}{\mathrm{Pr}}=b\gamma\ \ \mbox{if and only if}\ \ \delta-b\gamma\nu=0.
\end{align}
Due to the hypothesis that the last term $\Delta^2\psi_t$ of linear Blackstock's model disappears, Becker's assumption \eqref{Becker_Assum} allows us to implement a further factorization: the heat operator $\partial_t-\frac{\nu}{\mathrm{Pr}}\Delta$ and the viscoelastic damped wave operator $\partial_t^2-\Delta-\delta\Delta\partial_t$. It apparently diminishes some challenges to probe Blackstock's model. Simultaneously, the particular consideration of a monatomic perfect gas does not satisfy from physicists' viewpoints because its bulk viscosity coefficient is zero \cite{Pierce}, and the monatomic perfect gas situation is circumscribed in the reality, e.g. polyatomic gases are not explored. Hence, as mentioned in \cite{Kaltenbacher-Thalhammer-2018,Chen-Ikehata-Palmieri=2023}, the study for general Blackstock's model is vital and indispensable.
\medskip

\noindent\textbf{Mathematical background of Blackstock's model:}
Under the weakly nonlinear scheme \cite{Coulouvrat-1992}, the so-called \emph{substitution corollary} investigates the approximation $\Delta\psi\approx\psi_{tt}$ in the nonlinearity. Thus, it leads to the so-called \emph{Blackstock-Crighton equation} that is the linear Blackstock's model in \eqref{Non-linear_Blackstock_2} carrying the nonlinear terms (the one on the right)
\begin{align*}
	\frac{1}{c_0^2}\partial_{t}\left(\partial_t|\nabla\psi|^2+(\gamma-1)\psi_t\Delta\psi\right)\approx\frac{1}{c_0^2}\partial_t^2\left(|\nabla\psi|^2+\frac{\gamma-1}{2}|\psi_t|^2\right).
\end{align*}
 Concerning previous researches on the Blackstock-Crighton equation under Becker’s assumption, we refer the interested readers to \cite{Brunnhuber-Kaltenbacher-2014,Brunnhuber-2015,Brunnhuber-2016,Brunnhuber-Meyer-2016,Celik-Kyed-2019,Gambera-Lizama-Prokopczyk-2021} for Dirichlet or Neumann boundary value problems. Additionally, replacing Fourier's law of heat conduction by Cattaneo's law in the energy conservation of modeling, \cite{Liu-Qin-Zhang=2022} lately investigated global (in time) existence of small data Sobolev solution to Blackstock-Cattaneo model in the whole space, i.e. the thermoviscous flow with second sound phenomena. The pioneering work for Blackstock's model without Becker's assumption was done by \cite{Kaltenbacher-Thalhammer-2018} for the Dirichlet boundary problem, in particular, the vanishing thermal diffusivity limit (i.e. $\kappa\downarrow0$) in the weak sense was derived. Quite recently, the authors of \cite{Chen-Ikehata-Palmieri=2023} studied the Cauchy problem for Blackstock's model, where global (in time) existence, optimal estimates and asymptotic profiles of solutions are derived for higher-dimensional cases. To be specific, concerning linear Blackstock's model, the large-time asymptotic profile of first-order was obtained for $n\geqslant 3$, and singular limits with respect to the thermal diffusivity associated with first- and second-order profiles were completed. However, there are still some open questions for Blackstock's model in $\mb{R}^n$, for example, optimal estimates for lower spatial dimensions, and optimal higher-order asymptotic profiles (i.e. optimal leading terms). We will partly answer these questions in the present work.

\subsection{Main purpose of this paper}
Recalling from Subsection \ref{Sub-1.1} as well as in Table \ref{Table_1} that
\begin{align*}
	\delta=\nu\left(b+\frac{\gamma-1}{\mathrm{Pr}}\right)=b\nu +(\gamma-1)\kappa \ \ \mbox{and}\ \ \kappa=\frac{\nu}{\mathrm{Pr}},
\end{align*}
 the Cauchy problem for linear Blackstock's model in \eqref{Non-linear_Blackstock_2} can be written in a pellucid form  by
\begin{align}\label{Eq-Blackstock}
\begin{cases}
(\partial_t-\kappa\Delta)(\psi_{tt}-\Delta\psi-\delta\Delta\psi_t)+\kappa(\gamma-1)(b\nu-\kappa)\Delta^2\psi_t=0,&x\in\mb{R}^n,\ t>0,\\
\psi(0,x)=\psi_0(x),\ \psi_t(0,x)=\psi_1(x),\ \psi_{tt}(0,x)=\psi_2(x),&x\in\mb{R}^n,
\end{cases}
\end{align}
with $0<\kappa\ll \delta$, $\gamma\in(1,\frac{5}{3}]$, $b>0$ as well as $\nu>0$. We will not repeat these settings again in the
subsequent statements. As we talked about in preceding part of the text, the main property that we concern about is large-time asymptotic profiles. It should be stressed here that we do not assume Becker's assumption so that all derived results are valid in general.

Firstly, we will derive optimal estimates of the solution to \eqref{Eq-Blackstock} for all spatial dimensions, where the statement will be fixed in Theorem \ref{thm:2.1}. It improves the results in \cite[Theorems 2.1 and 2.3]{Chen-Ikehata-Palmieri=2023} in the lower dimensional cases by some effective representations formula and WKB analysis. More importantly, the core of this work is searching optimal leading term
\begin{align*}
	\ml{F}_{\xi\to x}^{-1}\left(\frac{1}{|\xi|^2}\left(\mathrm{e}^{-\kappa|\xi|^2t}-\cos(|\xi|t)\mathrm{e}^{-\frac{\delta}{2}|\xi|^2t}\right)\right)\int_{\mb{R}^n}\psi_2(x)\mathrm{d}x
\end{align*}
in Theorem \ref{thm:2.2}, in which the optimal upper bounds and lower bounds for $n\geqslant 1$ will be estimated under different hypotheses of initial datum. Indeed, the main difficulty that we need to overcome is the strong singularity as $|\xi|\to0$ for lower dimensions, for example, the next term with fixed $t_0>0$:
\begin{align*}
\left\|\ml{F}_{\xi\to x}^{-1}\left(\chi_{\intt}(\xi)\frac{1}{|\xi|^2}\mathrm{e}^{-\kappa|\xi|^2t_0}\right)\right\|_{L^2}^2=C\int_0^{\varepsilon_0}r^{n-5}\mathrm{e}^{-2\kappa r^2t_0}\mathrm{d}r\approx C\int_0^{\varepsilon_0}r^{n-5}\mathrm{d}r
\end{align*}
will bring some challengings when $n=1,\dots,4$. As a by-product, we will determine second-order asymptotic profile as higher-order diffusion-waves for $t\gg1$. Our proofs for all results will be given in Sections \ref{Section-3} and \ref{Section-4} step by step. Finally, some concluding remarks concerning about corresponding semilinear Cauchy problems in Section \ref{Section-5} complete the paper.

\subsection{Notations}\label{Sub-sec-notation}
Let us introduce some notations that will be used in this paper. We take the following zones of the Fourier space:
\begin{align*}
	\ml{Z}_{\intt}(\varepsilon_0):=\{|\xi|\leqslant\varepsilon_0\ll1\},\ \ 
	\ml{Z}_{\bdd}(\varepsilon_0,N_0):=\{\varepsilon_0\leqslant |\xi|\leqslant N_0\},\ \ 
	\ml{Z}_{\extt}(N_0):=\{|\xi|\geqslant N_0\gg1\}.
\end{align*}
Moreover, the cut-off functions $\chi_{\intt}(\xi),\chi_{\bdd}(\xi),\chi_{\extt}(\xi)\in \mathcal{C}^{\infty}$ having their supports in the corresponding zones $\ml{Z}_{\intt}(\varepsilon_0)$, $\ml{Z}_{\bdd}(\varepsilon_0/2,2N_0)$ and $\ml{Z}_{\extt}(N_0)$, respectively, such that
\begin{align*}
\chi_{\bdd}(\xi)=1-\chi_{\intt}(\xi)-\chi_{\extt}(\xi)\ \ \mbox{for all}\ \ \xi \in \mb{R}^n.
\end{align*}
The symbols of pseudo-differential operators $|D|$ and $\langle D\rangle $ are denoted by $|\xi|$ and $\langle \xi\rangle$, respectively, where $\langle \xi\rangle:=\sqrt{1+|\xi|^2}$ is  the Japanese bracket.

The symbol $f\lesssim g$ means: there exists a positive constant $C$ fulfilling $f\leqslant Cg$, which may be modified in different lines, analogously, for $f\gtrsim g$. Furthermore, the relation  $f\asymp g$ holds iff $f\lesssim g$ and $f\gtrsim g$ concurrently. We take the notation $\circ$ as the inner product in Euclidean space.

Let us recall the weighted $L^1$ space that is
\begin{align*}
	L^{1,1}:=\left\{f\in L^1 \ \big|\ \|f\|_{L^{1,1}}:=\int_{\mb{R}^n}(1+|x|)|f(x)|\mathrm{d}x<\infty \right\}
\end{align*}
so that $\|f\|_{L^1}\leqslant\|f\|_{L^{1,1}}$. The (weighted) mean of a summable function $f$ are denoted by
\begin{align*}
\mb{R}\ni P_f:=\int_{\mb{R}^n}f(x)\mathrm{d}x	\ \ \mbox{and}\ \ \mb{R}^n\ni M_f:=\int_{\mb{R}^n}xf(x)\mathrm{d}x.
\end{align*}
The Sobolev spaces of negative order $H^{-s}$ with $s>0$ are defined by $\ml{F}^{-1}(\langle \xi \rangle^{-s}\ml{F}(f))\in L^2$.

To end the introduction, we establish two time-dependent functions as follows:
\begin{align*}
\ml{D}_n^{(1)}(t):=\begin{cases}
t^{\frac{1}{2}}&\mbox{if}\ \ n=1,\\
(\ln t)^{\frac{1}{2}}&\mbox{if}\ \ n=2,\\
t^{\frac{1}{2}-\frac{n}{4}}&\mbox{if}\ \ n\geqslant 3,
\end{cases}
\ \ \mbox{and}\ \ \ml{D}_n^{(2)}(t):=
\begin{cases}
t^{2-\frac{n}{2}}&\mbox{if}\ \ n\leqslant 3,\\
(\ln t)^{\frac{1}{2}}&\mbox{if}\ \ n=4,\\
t^{1-\frac{n}{4}}&\mbox{if}\ \ n\geqslant5,
\end{cases}
\end{align*}
to be the growth or decay rates later.  It is worth noting that $n=2$ and $n=4$ are the critical dimensions to $\ml{D}_n^{(1)}(t)$ and $\ml{D}_n^{(2)}(t)$, respectively.

\section{Main results}
Before stating our main results in this work, let us introduce
\begin{align*}
\psi^{(1)}(t,x):=G_0(t,x)P_{\psi_2}
\end{align*}
as well as
\begin{align*}
\psi^{(2)}(t,x):=\nabla G_0(t,x)\circ M_{\psi_2}+H_0(t,x)P_{\psi_2}+G_1(t,x)\left(P_{\psi_1}+\frac{2\kappa-\delta}{2}P_{\psi_2}\right),
\end{align*}
carrying some functions (by inverse Fourier transform)
\begin{align*}
G_0(t,x)&:=\ml{F}_{\xi\to x}^{-1}\left(\frac{1}{|\xi|^2}\left(\mathrm{e}^{-\kappa|\xi|^2t}-\cos(|\xi|t)\mathrm{e}^{-\frac{\delta}{2}|\xi|^2t}\right)\right),\\
G_1(t,x)&:=\ml{F}^{-1}_{\xi\to x}\left(\frac{1}{|\xi|}\sin(|\xi|t)\mathrm{e}^{-\frac{\delta}{2}|\xi|^2t}\right),\\
H_0(t,x)&:=\frac{4\hat{\delta}-\kappa\delta^2}{8\kappa}t\,\ml{F}_{\xi\to x}^{-1}\left(|\xi|\sin(|\xi|t)\mathrm{e}^{-\frac{\delta}{2}|\xi|^2t}\right)
=-\frac{4\hat{\delta}-\kappa\delta^2}{8\kappa}t \Delta G_1(t,x),
\end{align*}
where the constant $\hat{\delta}:=-\kappa^2(\delta-\gamma b\nu)$
and again ``$\circ$'' represents the inner product in $\mathbb{R}^{n}$. Henceforth, the special case with Becker's assumption will be encompassed within all subsequent results for general circumstance.
%

We now formulate our first result, which states the first-order asymptotic profile of the solution to \eqref{Eq-Blackstock} and the sharpness of upper bound in the $L^{2}$ norm for large-time. Particularly, it implies infinite time $L^2$-blow-up of the solution if and only if $n=1,2,3$ with polynomial growth and $n=4$ with logarithmic growth.
\begin{theorem} \label{thm:2.1}
Let us consider the Cauchy problem \eqref{Eq-Blackstock} with initial datum $\psi_j\in H^{-2j}\cap L^{1}$ for any $n\geqslant 1$, where $j=0,1,2$. Then, the solution satisfies the following refined estimate for $t\gg1$: 
\begin{align}\label{eq:6}
	\|\psi(t,\cdot)-\psi^{(1)}(t,\cdot)\|_{L^2} = o\big( \ml{D}_n^{(2)}(t) \big).  
\end{align} 
Moreover, providing that $|P_{\psi_2}|$ does not vanish, then the solution satisfies the following optimal estimate for $t\gg1$:
\begin{align}\label{Opt-Est}
	\ml{D}_n^{(2)}(t)|P_{\psi_2}|\lesssim\|\psi(t,\cdot)\|_{L^2}\lesssim \ml{D}_{n}^{(2)}(t)\sum\limits_{j=0,1,2}\|\psi_j\|_{H^{-2j}\cap L^{1}}.
\end{align}
\end{theorem}
\begin{remark}\label{Rem-prof}
%
The refined estimate \eqref{eq:6} means that the solution $\psi(t,\cdot)$ behaves like $\psi^{(1)}(t,\cdot)$ as $t \to \infty$ in the $L^2$ framework, which gives the reason why we call $\psi^{(1)}(t,x)$ as a first-order profile.
Here, the function $\psi^{(1)}(t,x)$ can be regarded as a linear combination of the diffusion-waves 
\begin{align*}
\ml{F}^{-1}_{\xi\to x}\left(
\cos(|\xi|t)
\mathrm{e}^{-\frac{\delta}{2}|\xi|^2t}
\right)
\end{align*}
and the heat kernel %
\begin{align*}
	\ml{F}^{-1}_{\xi\to x}\left(
	\mathrm{e}^{-\kappa |\xi|^2t}
	\right)
	=  
	\frac{1}{(4 \pi\kappa t)^{n/2}}
	\mathrm{e}^{-\frac{|x|^{2}}{4 \kappa t}}
\end{align*}
with the singularity $|\xi|^{-2}$ near $|\xi| =0$.
This fact indicates that the desired estimates in the low dimensional cases, especially $n=1,2$, are delicate, when we deal with them in the $L^{2}$ framework.   
\end{remark}
\begin{remark}
By taking the same assumption as those in \cite[Theorems 2.1 and 2.3]{Chen-Ikehata-Palmieri=2023}, the derived estimate \eqref{Opt-Est} fills the crucial gap (for upper bound and lower bound also) in lower dimensions $n=1,2$. The difficulty of not summable singularity in $|\xi|=0$ stated in \cite[Remark 2.6]{Chen-Ikehata-Palmieri=2023} is solved. In other words, a complete optimal estimate of $\psi(t,\cdot)$ in the $L^2$ norm such that
\begin{align*}
\|\psi(t,\cdot)\|_{L^2}\asymp \ml{D}_n^{(2)}(t)\ \ \mbox{for}\ \  t\gg1 ,
\end{align*}
 is derived for all dimensions with both $L^{1}$ and weighted $L^1$ datum.
\end{remark}
In the forthcoming discussion, we will claim the second-order approximation of the solution as $t \to \infty$, which enables us to discuss the optimality of the first-order asymptotic profile (i.e. the leading term). Particularly, it implies infinite time $L^2$-blow-up of the error term (by subtracting the first-order profile) if and only if $n=1,2$.
For the sake of the simplicity, we take the notations 
\begin{align*}
P_{\Psi_{1,2}} := P_{\psi_1}+\frac{2\kappa-\delta}{2}P_{\psi_2} \ \ \mbox{and} \ \ P_{2} := \frac{4\hat{\delta}-\kappa\delta^2}{8\kappa}P_{\psi_2}
\end{align*}
for suitable combined datum.
\begin{theorem} \label{thm:2.2}
Under the same assumption on initial datum as those in Theorem \ref{thm:2.1}, let us suppose that $\psi_{2} \in L^{1,1}$ additionally.
Then, the next refined estimates hold for $t\gg1$:
\begin{align}\label{Upper-second}
	\|\psi(t,\cdot)-\psi^{(1)}(t,\cdot)-\psi^{(2)}(t,\cdot)\|_{L^2}= o\big( \ml{D}_n^{(1)}(t) \big)
\end{align}
as well as
\begin{align}\label{Opt-Prof}
	\|\psi(t,\cdot)-\psi^{(1)}(t,\cdot)\|_{L^2}\lesssim\ml{D}_n^{(1)}(t) \left( \sum\limits_{j=0,1}\|\psi_j\|_{H^{-2j} \cap L^{1}} + \|\psi_2 \|_{H^{-4} \cap L^{1,1}} \right).
\end{align}
Moreover, the lower bound estimates can be shown in different dimensions.
\begin{description}
	\item[(I) Lower dimensional cases: $n=1,2$.] Providing that $|M_{\psi_2}|$ and $|P_{\Psi_{1,2}}|$ do not vanish simultaneously, then the following sharp lower bound estimate holds for $t\gg1$:  
	\begin{align}\label{Opt-Prof1}
		\left(|P_{\Psi_{1,2}}| + |M_{\psi_2}|\right) \ml{D}_n^{(1)}(t) \lesssim\|\psi(t,\cdot)-\psi^{(1)}(t,\cdot)\|_{L^2}.
	\end{align}
\item[(II)  Higher dimensional cases: $n\geqslant 3$.] Providing that $|M_{\psi_2}|$ and $|2\delta P_{\Psi_{1,2}}+nP_2|$ do not vanish simultaneously, then the following sharp lower bound estimate holds for $t\gg1$:
\begin{align}\label{eq:12}
	\left(|2\delta P_{\Psi_{1,2}}+nP_2|+ |M_{\psi_2}| \right) \ml{D}_n^{(1)}(t) \lesssim\|\psi(t,\cdot)-\psi^{(1)}(t,\cdot)\|_{L^2}.
\end{align}

\end{description}

\end{theorem}


\begin{remark}
By subtracting the second-order asymptotic profile $\psi^{(2)}(t,\cdot)$ in the $L^2$ norm, we arrive at a faster decay estimate \eqref{Upper-second} in the comparison with \eqref{Opt-Prof}, which somehow implies generalized diffusion phenomena of second-order. As we mentioned in Remark \ref{Rem-prof}, the second-order profile may be understood by higher-order diffusion-waves.
\end{remark}
\begin{remark}
Due to the limit 
\begin{align*}
\lim\limits_{t\to\infty}\frac{\ml{D}_n^{(2)}(t)}{\ml{D}_n^{(1)}(t)}=\infty\ \ \mbox{for any}\ \ n\geqslant 1,
\end{align*}
we ensure some improvements as follows: $t^{-1}$ when $n=1$; $(\ln t)^{\frac{1}{2}}t^{-1}$ when $n=2$; $t^{-\frac{3}{4}}$ when $n=3$; $(\ln t)^{-\frac{1}{2}}t^{-\frac{1}{2}}$ when $n=4$; $t^{-\frac{1}{2}}$ when $n\geqslant 5$; by subtracting the first-order profile $\psi^{(1)}(t,\cdot)$ in the $L^2$ norm for $t\gg1$. Among them, the one-dimensional case promises the greatest gain in decay rates.
\end{remark}
\begin{remark}
The last result offers an alternative to obtain the sharp lower bound estimates of the first-order approximation $\|\psi(t,\cdot)-\psi^{(1)}(t,\cdot)\|_{L^2}$, 
even if $|P_{\psi_0}|=|P_{\psi_1}|=|P_{\psi_2}|= 0$.
It is also worth pointing out that the situations are different between $n=1,2$ and $n\geqslant 3$.
This phenomenon originates from the lower bound estimates of $\psi^{(2)}(t,\cdot)$ in the $L^2$ norm, that is, 
the $L^{2}$ integrability of $|\xi|^{-1}$ near $|\xi|=0$ in $\mathbb{R}^{n}$.
Roughly speaking, when $n=1,2$, we have to use the oscillation of $\sin(|\xi| t)$ in $\widehat{\psi}^{(2)}(t,\xi)$ to recover the $L^{2}$ integrability of $|\xi|^{-1}$ and to derive sharp decay properties, 
while $|\xi|^{-1}$ is $L^{2}$ integrable near $|\xi|=0$ for any $n\geqslant 3$.
\end{remark}
\begin{remark}
According to the derived estimates \eqref{Opt-Prof}-\eqref{eq:12}, we have obtained the optimal leading term $\psi^{(1)}(t,x)$ for any $n\geqslant 1$ already in the sense that\begin{align*}
\|\psi(t,\cdot)-\psi^{(1)}(t,\cdot)\|_{L^2}\asymp \ml{D}_n^{(1)}(t) \ \ \mbox{for}\ \ t\gg1,
\end{align*}
which is one of the innovation of this paper. Different from \cite[Theorem 2.2]{Chen-Ikehata-Palmieri=2023}, our optimal leading term excludes the superfluous terms with
\begin{align*}
	\ml{F}^{-1}_{\xi\to x}\left(\mathrm{e}^{-\kappa|\xi|^2t}\right)\ \ \mbox{and}\ \ \ml{F}^{-1}_{\xi\to x}\left(\frac{1}{|\xi|}\sin(|\xi|t)\mathrm{e}^{-\frac{\delta}{2}|\xi|^2t}\right),
\end{align*}
due to some applications of refined WKB analysis.
\end{remark}

\begin{remark}
In the simplest case with Becker's assumption, we may derive higher-order asymptotic profiles of solution by an easy way: firstly, we may set up with the aid of $\varphi^{\mathrm{h}}:=\psi_t-\kappa\Delta\psi$ so that the model will be reduced to
\begin{align*}
	\varphi^{\mathrm{h}}_{tt}-\Delta\varphi^{\mathrm{h}}-\delta\Delta\varphi^{\mathrm{h}}_t=0,\ \ x\in\mb{R}^n,\ t>0,
\end{align*}
with suitable datum $\varphi^{\mathrm{h}}(0,x)=\psi_1(x)-\kappa\Delta\psi_0(x)$ as well as $\varphi^{\mathrm{h}}_t(0,x)=\psi_2(x)-\kappa\Delta\psi_1(x)$; secondly, from some known results (for example, Theorems 3.1, 3.2 and 3.3 in \cite{Michihisa=2018} or \cite{Michihisa=2021}), we may describe asymptotic behaviors for $\varphi^{\mathrm{h}}=\varphi^{\mathrm{h}}(t,x)$; eventually, from the backward transform
\begin{align*}
	\psi(t,x)=\mathrm{e}^{\kappa\Delta t}\psi_0(x)+\int_0^t\mathrm{e}^{\kappa \Delta (t-s)}\varphi^{\mathrm{h}}(s,x)\mathrm{d}s
\end{align*}
to achieve our aim.
\end{remark}

\section{Asymptotic behaviors of solution in the Fourier space}\label{Section-3}
As preparations for deriving large-time asymptotic profiles for the linear Cauchy problem \eqref{Eq-Blackstock}, this section mainly concentrates on refined behaviors of solution to Blackstock's model in the Fourier space, namely,
\begin{align}\label{Eq-Blackstock-Fourier}
\begin{cases}
\widehat{\psi}_{ttt}+(\delta+\kappa)|\xi|^2\widehat{\psi}_{tt}+(1+\tilde{\gamma}|\xi|^2)|\xi|^2\widehat{\psi}_t+\kappa|\xi|^4\widehat{\psi}=0,&\xi\in\mb{R}^n,\ t>0,\\
\widehat{\psi}(0,\xi)=\widehat{\psi}_0(\xi),\ \widehat{\psi}_t(0,\xi)=\widehat{\psi}_1(\xi),\ \widehat{\psi}_{tt}(0,\xi)=\widehat{\psi}_2(\xi),&\xi\in\mb{R}^n,
\end{cases}
\end{align}
where the partial Fourier transform with respect to spatial variables was employed, and we took $\tilde{\gamma}:=\gamma b\nu\kappa$ for the simplicity of notations. The corresponding characteristic equation to the last model \eqref{Eq-Blackstock-Fourier}, therefore, is given by
\begin{align}\label{Cubic}
	\lambda^3+(\delta+\kappa)|\xi|^2\lambda^2+(1+\tilde{\gamma}|\xi|^2)|\xi|^2\lambda+\kappa|\xi|^4=0.
\end{align}

Recalling the derived pointwise estimates in the restricted zones, the authors of \cite{Chen-Ikehata-Palmieri=2023} proved
\begin{align}\label{Estimate-large-middle}
\big(1-\chi_{\intt}(\xi)\big)|\widehat{\psi}(t,\xi)|\lesssim\big(1-\chi_{\intt}(\xi)\big)\mathrm{e}^{-ct}\left(|\widehat{\psi}_0(\xi)|+\langle\xi\rangle^{-2}|\widehat{\psi}_1(\xi)|+\langle\xi\rangle^{-4}|\widehat{\psi}_2(\xi)|\right)
\end{align}
with a positive constant $c$. It tells us that concerning $\xi\in\ml{Z}_{\bdd}(\varepsilon_0,N_0)\cup\ml{Z}_{\extt}(N_0)$ the solution in the Fourier space decays exponentially with suitable regular datum. In other words, the growth or decay rate of solution to Blackstock's model is determined by frequencies localizing in the small zone $\xi\in\ml{Z}_{\intt}(\varepsilon_0)$. The remaining part of this section devotes to asymptotic analysis for $\chi_{\intt}(\xi)\widehat{\psi}(t,\xi)$. So, we will expand characteristic roots in the first step.

\subsection{Asymptotic expansions of characteristic roots for small frequencies}
Indeed, \cite[Subsection 2.1]{Chen-Ikehata-Palmieri=2023} has already obtained
\begin{align}\label{C-I-P-result}
	\lambda_1=-\kappa|\xi|^2+\ml{O}(|\xi|^3)\ \ \mbox{and}\ \ \lambda_{2,3}=\pm i|\xi|-\frac{\delta}{2}|\xi|^2+\ml{O}(|\xi|^3)
\end{align}
for $\xi\in\ml{Z}_{\intt}(\varepsilon_0)$ by Taylor-like expansions up to $o(|\xi|^2)$. Nevertheless, in order to deeply study large-time profiles, particularly second-order profiles of solution, we will employ a different method (associated with explicit cubic roots) comparing with \cite{Chen-Ikehata-Palmieri=2023}.

According to the basic algebra, let us denote
\begin{align*}
P:=|\xi|^2+\left(\tilde{\gamma}-\frac{(\delta+\kappa)^2}{3}\right)|\xi|^4 \ \ \mbox{and}\ \ Q:=\frac{2\kappa-\delta}{3}|\xi|^4+\left(\frac{2(\delta+\kappa)^3}{27}-\frac{(\delta+\kappa)\tilde{\gamma}}{3} \right)|\xi|^6,
\end{align*}
and the discriminant of the cubic \eqref{Cubic} is
\begin{align*}
	\triangle_{\mathrm{Dis}}=-4P^3-27Q^2=-4|\xi|^6+(\delta^2+20\kappa\delta-8\kappa^2-12\tilde{\gamma})|\xi|^8+\ml{O}(|\xi|^{10})<0,
\end{align*} 
as $\xi\in\ml{Z}_{\intt}(\varepsilon_0)$, which implies that the cubic \eqref{Cubic} owns one real root $\lambda_1$ and two complex (non-real) conjugate roots $\lambda_{2,3}=\lambda_{\mathrm{R}}\pm i\lambda_{\mathrm{I}}$ for small frequencies. The real root can be effectively expressed by Cardano's formula as follows:
\begin{align}\label{Lambda1}
	\lambda_1=-\frac{\delta+\kappa}{3}|\xi|^2+\sum\limits_{\pm}\sqrt[3]{-\frac{1}{2}Q \pm\frac{1}{6\sqrt{3}}\sqrt{-\triangle_{\mathrm{Dis}} }}
\end{align}
for any $\xi\in\ml{Z}_{\intt}(\varepsilon_0)$ due to the negative discriminant. Vieta's formula says
\begin{align}\label{Algebraic}
\begin{cases}
\lambda_1+2\lambda_{\mathrm{R}}=-(\delta+\kappa)|\xi|^2,\\
2\lambda_1\lambda_{\mathrm{R}}+\lambda_{\mathrm{R}}^2+\lambda_{\mathrm{I}}^2=(1+\tilde{\gamma}|\xi|^2)|\xi|^2,\\
\lambda_1(\lambda_{\mathrm{R}}^2+\lambda_{\mathrm{I}}^2)=-\kappa|\xi|^4.
\end{cases}
\end{align}
The linear algebraic equation \eqref{Algebraic} associated with \eqref{Lambda1} can determine $\lambda_{\mathrm{R}},\lambda_{\mathrm{I}}\in\mb{R}$ easily.

Let us recall
\begin{align*}
(1+z)^{\frac{1}{3}}=1+\frac{z}{3}-\frac{z^2}{9}+\frac{5z^3}{81}+\ml{O}(z^4)\ \ \mbox{for}\  \ |z|\ll 1.
\end{align*}
Denoting $4\tilde{\delta}:=\delta^2+20\kappa\delta-8\kappa^2-12\tilde{\gamma}$, due to the expansions
\begin{align*}
\sqrt{-\triangle_{\mathrm{Dis}} }&=\sqrt{4|\xi|^6-4\tilde{\delta}|\xi|^8+\ml{O}(|\xi|^{10})}=2|\xi|^3\sqrt{1-\tilde{\delta}|\xi|^2+\ml{O}(|\xi|^{4})}\\
&=2|\xi|^3-\tilde{\delta}|\xi|^5+\ml{O}(|\xi|^7),
\end{align*}
we immediately get
\begin{align*}
\lambda_1+\frac{\delta+\kappa}{3}|\xi|^2&=\sum\limits_{\pm}\pm\frac{|\xi|}{\sqrt{3}}\sqrt[3]{1\mp\frac{\sqrt{3}(2\kappa-\delta)}{2}|\xi|-\frac{\tilde{\delta}}{2}|\xi|^2\mp\sqrt{3}\left(\frac{(\delta+\kappa)^3}{9}-\frac{(\delta+\kappa)\tilde{\gamma}}{2}\right)|\xi|^3+\ml{O}(|\xi|^4)}\\
&=\sum\limits_{\pm}\pm\frac{|\xi|}{\sqrt{3}}\left(1\mp\frac{\sqrt{3}(2\kappa-\delta)}{6}|\xi|-\frac{\tilde{\delta}}{6}|\xi|^2\mp\frac{\sqrt{3}}{3}\left(\frac{(\delta+\kappa)^3}{9}-\frac{(\delta+\kappa)\tilde{\gamma}}{2}\right)|\xi|^3\right.\\
&\qquad\qquad\qquad\left.-\frac{(2\kappa-\delta)^2}{12}|\xi|^2\mp\frac{\sqrt{3}(2\kappa-\delta)\tilde{\delta}}{18}|\xi|^3\mp\frac{5\sqrt{3}(2\kappa-\delta)^3}{27}|\xi|^3+\ml{O}(|\xi|^4)\right)\\
&=-\frac{2\kappa-\delta}{3}|\xi|^2+\hat{\delta}|\xi|^4+\ml{O}(|\xi|^5),
\end{align*}
where we defined \begin{align*}
\hat{\delta}&=-\frac{2(\delta+\kappa)^3}{27}+\frac{(\delta+\kappa)\tilde{\gamma}}{3}-\frac{(2\kappa-\delta)\tilde{\delta}}{9}-\frac{5(2\kappa-\delta)^3}{108}\\
&=-\kappa^2\delta+\kappa\tilde{\gamma}=-\kappa^2(\delta-\gamma b\nu).
\end{align*} Thus, we have $\lambda_1=-\kappa|\xi|^2+\hat{\delta}|\xi|^4+\ml{O}(|\xi|^5)$. The first equation of \eqref{Algebraic} indicates 
\begin{align*}
\lambda_{\mathrm{R}}=-\frac{\delta}{2}|\xi|^2-\frac{\hat{\delta}}{2}|\xi|^4+\ml{O}(|\xi|^5),
\end{align*} and the third one yields
\begin{align*}
\lambda_{\mathrm{I}}&=\sqrt{-\frac{\kappa|\xi|^4}{\lambda_1}-\lambda_{\mathrm{R}}^2}=\sqrt{\frac{|\xi|^2}{1-\frac{\hat{\delta}}{\kappa}|\xi|^2+\ml{O}(|\xi|^3)}-\frac{\delta^2}{4}|\xi|^4+\ml{O}(|\xi|^6)}\\
&=|\xi|+\frac{4\hat{\delta}-\kappa\delta^2}{8\kappa}|\xi|^3+\ml{O}(|\xi|^5).
\end{align*}

Then, summarizing the last discussion, we claim the next result.
\begin{prop}\label{Prop-Char-root}
	Let $\xi\in\ml{Z}_{\intt}(\varepsilon_0)$ with $0<\varepsilon_0\ll 1$. Then, the roots $\lambda_j$ for $j=1,2,3$ to the cubic \eqref{Cubic} can be expanded by
\begin{align*}
\lambda_1&=-\kappa|\xi|^2+\hat{\delta}|\xi|^4+\ml{O}(|\xi|^5),\\
\lambda_{2,3}&=\pm i|\xi|-\frac{\delta}{2}|\xi|^2\pm i\frac{4\hat{\delta}-\kappa\delta^2}{8\kappa}|\xi|^3-\frac{\hat{\delta}}{2}|\xi|^4+\ml{O}(|\xi|^5),
\end{align*}
where the real-valued constant $\hat{\delta}=-\kappa^2(\delta-\gamma b\nu)$. Additionally, last two roots fulfill $\mathrm{Re}\lambda_2=\mathrm{Re}\lambda_3$ and $\mathrm{Im}\lambda_2+\mathrm{Im}\lambda_3=0$.
\end{prop}
\begin{remark}
In the comparison with those in \cite{Chen-Ikehata-Palmieri=2023}, i.e. the expressions \eqref{C-I-P-result}, the higher-order expansions of three characteristic roots are provided in Proposition \ref{Prop-Char-root}. The explicit representations in terms of $|\xi|^3$ and $|\xi|^4$ will contribute to improved decay rate when we consider second-order asymptotic profiles.
\end{remark}
\subsection{Asymptotic profiles of kernels for small frequencies}
From pairwise distinct characteristic roots derived in Proposition \ref{Prop-Char-root}, the solution to the Cauchy problem \eqref{Eq-Blackstock-Fourier} when $\xi\in\ml{Z}_{\intt}(\varepsilon_0)$ can be represented via
\begin{align}\label{Representation_Fourier_u}
	\widehat{\psi} (t,\xi)=\widehat{K}_0(t,|\xi|)\widehat{\psi}_0 (\xi)+\widehat{K}_1(t,|\xi|)\widehat{\psi}_1 (\xi)+\widehat{K}_2(t,|\xi|)\widehat{\psi}_2 (\xi),
\end{align}
where the kernels in the Fourier space have the representations
\begin{align*}
	\widehat{K}_0(t,|\xi|)&:=\sum\limits_{j=1,2,3}\frac{\exp(\lambda_jt)\prod_{k=1,2,3,\ k\neq j}\lambda_k}{\prod_{k=1,2,3,\ k\neq j}(\lambda_j-\lambda_k)},\\
	\widehat{K}_1(t,|\xi|)&:=-\sum\limits_{j=1,2,3}\frac{\exp(\lambda_jt)\sum_{k=1,2,3,\ k\neq j}\lambda_k}{\prod_{k=1,2,3,\ k\neq j}(\lambda_j-\lambda_k)},\\
	\widehat{K}_2(t,|\xi|)&:=\sum\limits_{j=1,2,3}\frac{\exp(\lambda_jt)}{\prod_{k=1,2,3,\ k\neq j}(\lambda_j-\lambda_k)},
\end{align*}
where $\lambda_j$ were given in Proposition \ref{Prop-Char-root}.

To understand the decay properties gradually, we next propose sharper estimates for the kernels, especially for $\chi_{\intt}(\xi)\widehat{K}_2(t,|\xi|)$, than those in \cite[Part IV]{Chen-Ikehata-Palmieri=2023}.
\begin{prop}\label{Prop-Zero-Order}
Let $\xi\in\ml{Z}_{\intt}(\varepsilon_0)$ with $0<\varepsilon_0\ll 1$. Then, the kernels fulfill the following pointwise estimates in the Fourier space:
\begin{align*}
\chi_{\intt}(\xi)|\widehat{K}_0(t,|\xi|)|&\lesssim\chi_{\intt}(\xi)\mathrm{e}^{-c|\xi|^2t},\\
\chi_{\intt}(\xi)|\widehat{K}_1(t,|\xi|)|&\lesssim\chi_{\intt}(\xi)\left(1+\frac{|\sin(|\xi|t)|}{|\xi|}\right)\mathrm{e}^{-c|\xi|^2t},\\
\chi_{\intt}(\xi)|\widehat{K}_2(t,|\xi|)|&\lesssim\chi_{\intt}(\xi)\left(|\cos(|\xi|t)|t+\frac{|\sin(|\xi|t)|}{|\xi|}+\frac{|\sin(\frac{|\xi|}{2}t)|^2}{|\xi|^2} \right)\mathrm{e}^{-c|\xi|^2t},
\end{align*}
with some constants $c>0$.
\end{prop}
\begin{remark}
In the estimate of $\chi_{\intt}(\xi)|\widehat{K}_2(t,|\xi|)|$, we improved the one in \cite{Chen-Ikehata-Palmieri=2023} by recovering the factor $|\sin(|\xi|t)|$ in the second term, which allows us to compensate a singularity as $|\xi|\to0$ for lower dimensions.
\end{remark}
\begin{proof}
The first two estimates have been derived in \cite[Part IV]{Chen-Ikehata-Palmieri=2023} already, and we just need to demonstrate the last one, which needs careful analysis in the Fourier space. We notice that
\begin{align}\label{P0}
P_0=(\lambda_{\mathrm{R}}-\lambda_1)^2+\lambda_{\mathrm{I}}^2=|\xi|^2+\ml{O}(|\xi|^4)\ \ \mbox{for}\ \ \xi\in\ml{Z}_{\intt}(\varepsilon_0).
\end{align}
 Benefit from some adaptable structures of characteristic roots, we rewrite the third kernel by
\begin{align*}
\widehat{K}_2(t,|\xi|)=\frac{1}{P_0}\left(\mathrm{e}^{\lambda_1t}-\cos(\lambda_{\mathrm{I}}t)\mathrm{e}^{\lambda_{\mathrm{R}}t}+\frac{\lambda_{\mathrm{R}}-\lambda_{\mathrm{1}}}{\lambda_{\mathrm{I}}}\sin(\lambda_{\mathrm{I}}t)\mathrm{e}^{\lambda_{\mathrm{R}}t}\right).
\end{align*}
We propose the fundamental trick
\begin{align*}
	\mathrm{e}^{\lambda_1t}-\cos(\lambda_{\mathrm{I}}t)\mathrm{e}^{\lambda_{\mathrm{R}}t}&=\mathrm{e}^{\lambda_1t}\big(1-\cos(\lambda_{\mathrm{I}}t)\big)+\cos(\lambda_{\mathrm{I}}t)\left(\mathrm{e}^{\lambda_1t}-\mathrm{e}^{\lambda_{\mathrm{R}}t}\right)\\
	&=2\mathrm{e}^{\lambda_1t}\sin^2\left(\tfrac{\lambda_{\mathrm{I}}}{2}t\right)-\cos(\lambda_{\mathrm{I}}t )\mathrm{e}^{\lambda_1t}(\lambda_{\mathrm{R}}-\lambda_{\mathrm{1}})t\int_0^1\mathrm{e}^{(\lambda_{\mathrm{R}}-\lambda_{\mathrm{1}})t\tau}\mathrm{d}\tau.
\end{align*}
To sum up, we arrive at
\begin{align*}
	\chi_{\intt}(\xi)|\widehat{K}_2(t,|\xi|)|&\lesssim\chi_{\intt}(\xi)\frac{\sin^2\left(\frac{\lambda_\mathrm{I}}{2}t\right)}{P_0}\mathrm{e}^{\lambda_1t}+\chi_{\intt}(\xi)\frac{|\lambda_{\mathrm{R}}-\lambda_{\mathrm{1}}|}{P_0|\lambda_{\mathrm{I}}|}\left|\sin(\lambda_{\mathrm{I}}t)\right|\mathrm{e}^{\lambda_{\mathrm{R}}t}\\
	&\quad+\chi_{\intt}(\xi)\left|\cos(\lambda_{\mathrm{I}}t)\right|\frac{|\lambda_{\mathrm{R}}-\lambda_{\mathrm{1}}|}{P_0}t\mathrm{e}^{\max\{\lambda_1,\lambda_{\mathrm{R}}\}t}.
\end{align*}
Finally, by plugging the characteristic roots stated in Proposition \ref{Prop-Char-root} as well as \eqref{P0}, we complete the proof.
\end{proof}

Rearranging the representation of $\widehat{\psi}(t,\xi)$ for small frequencies, we are able to express
\begin{align*}
	\widehat{\psi}(t,\xi)&=\frac{-(\lambda_{\mathrm{I}}^2+\lambda_{\mathrm{R}}^2)\widehat{\psi}_0+2\lambda_{\mathrm{R}}\widehat{\psi}_1-\widehat{\psi}_2}{2\lambda_{\mathrm{R}}\lambda_1-\lambda_{\mathrm{I}}^2-\lambda_{\mathrm{R}}^2-\lambda_1^2}\mathrm{e}^{\lambda_1t}+\frac{(2\lambda_{\mathrm{R}}\lambda_1-\lambda_1^2)\widehat{\psi}_0-2\lambda_{\mathrm{R}}\widehat{\psi}_1+\widehat{\psi}_2}{2\lambda_{\mathrm{R}}\lambda_1-\lambda_{\mathrm{I}}^2-\lambda_{\mathrm{R}}^2-\lambda_1^2}\cos(\lambda_{\mathrm{I}}t)\mathrm{e}^{\lambda_{\mathrm{R}}t}\\
	&\quad+\frac{\lambda_1(\lambda_{\mathrm{R}}\lambda_1+\lambda_{\mathrm{I}}^2-\lambda_{\mathrm{R}}^2)\widehat{\psi}_0+(\lambda_{\mathrm{R}}^2-\lambda_{\mathrm{I}}^2-\lambda_1^2)\widehat{\psi}_1-(\lambda_{\mathrm{R}}-\lambda_1)\widehat{\psi}_2}{\lambda_{\mathrm{I}}(2\lambda_{\mathrm{R}}\lambda_1-\lambda_{\mathrm{I}}^2-\lambda_{\mathrm{R}}^2-\lambda_1^2)}\sin(\lambda_{\mathrm{I}}t)\mathrm{e}^{\lambda_{\mathrm{R}}t}.
\end{align*}
Motivated by asymptotic expansions of characteristic roots, to extract the dominant parts we may introduce two auxiliary functions in the Fourier space as follows:
\begin{align*}
\widehat{J}_0(t,\xi)&:=\frac{\cos(\lambda_{\mathrm{I}} t)\mathrm{e}^{\lambda_{\mathrm{R}}t}-\mathrm{e}^{\lambda_1t}}{2\lambda_{\mathrm{R}}\lambda_1-\lambda_{\mathrm{I}}^2-\lambda_{\mathrm{R}}^2-\lambda_1^2}\widehat{\psi}_2,\\
\widehat{J}_1(t,\xi)&:=\frac{-\lambda_{\mathrm{I}}^2\sin(\lambda_{\mathrm{I}}t)\mathrm{e}^{\lambda_{\mathrm{R}}t}}{\lambda_{\mathrm{I}}(2\lambda_{\mathrm{R}}\lambda_1-\lambda_{\mathrm{I}}^2-\lambda_{\mathrm{R}}^2-\lambda_1^2)}\widehat{\psi}_1+\frac{-(\lambda_{\mathrm{R}}-\lambda_1)\sin(\lambda_{\mathrm{I}}t)\mathrm{e}^{\lambda_{\mathrm{R}}t}}{\lambda_{\mathrm{I}}(2\lambda_{\mathrm{R}}\lambda_1-\lambda_{\mathrm{I}}^2-\lambda_{\mathrm{R}}^2-\lambda_1^2)}\widehat{\psi}_2.
\end{align*}
Some straightforward subtractions conclude the next proposition easily.
\begin{prop}\label{Prop-3.3}
Let $\xi\in\ml{Z}_{\intt}(\varepsilon_0)$ with $0<\varepsilon_0\ll 1$. Then, the solution to the Cauchy problem \eqref{Eq-Blackstock-Fourier} fulfills the following refined estimates in the Fourier space:
\begin{align*}
	\chi_{\intt}(\xi)|\widehat{\psi}(t,\xi)-\widehat{J}_0(t,\xi)|&\lesssim \chi_{\intt}(\xi)\mathrm{e}^{-c|\xi|^2t}\left(|\widehat{\psi}_0|+\frac{|\sin(|\xi|t)|}{|\xi|}(|\widehat{\psi}_1|+|\widehat{\psi}_2|)\right),\\
	\chi_{\intt}(\xi)|\widehat{\psi}(t,\xi)-\widehat{J}_0(t,\xi)-\widehat{J}_1(t,\xi)|&\lesssim \chi_{\intt}(\xi)\mathrm{e}^{-c|\xi|^2t}(|\widehat{\psi}_0|+|\widehat{\psi}_1|),
\end{align*}
with some constants $c>0$.
\end{prop}

Let us construct the following symbols related to diffusion-waves by ignoring higher-order terms in $\widehat{J}_0(t,\xi)$ and $\widehat{J}_1(t,\xi)$, respectively:
\begin{align}\label{Sym-01}
\widehat{G}_0(t,|\xi|)=\frac{1}{|\xi|^2}\left(\mathrm{e}^{-\kappa|\xi|^2t}-\cos(|\xi|t)\mathrm{e}^{-\frac{\delta}{2}|\xi|^2t}\right)\ \ \mbox{and}\  \ \widehat{G}_1(t,|\xi|)=\frac{1}{|\xi|}\sin(|\xi|t)\mathrm{e}^{-\frac{\delta}{2}|\xi|^2t}.
\end{align}
Furthermore, we introduce the auxiliary function
\begin{align}\label{Sym-02}
\widehat{H}_0(t,|\xi|)=\frac{4\hat{\delta}-\kappa\delta^2}{8\kappa}|\xi|t\sin(|\xi|t)\mathrm{e}^{-\frac{\delta}{2}|\xi|^2t}.
\end{align}
\begin{prop}\label{Prop-3.4}
Let $\xi\in\ml{Z}_{\intt}(\varepsilon_0)$ with $0<\varepsilon_0\ll 1$. Then, the auxiliary functions fulfill the following refined estimates in the Fourier space:
\begin{align}\label{Est-01}
	\chi_{\intt}(\xi)|\widehat{J}_0(t,\xi)-\widehat{G}_0(t,|\xi|)\widehat{\psi}_2|&\lesssim \chi_{\intt}(\xi)t^{\frac{1}{2}}\mathrm{e}^{-c|\xi|^2t}|\widehat{\psi}_2|,\\
	\chi_{\intt}(\xi)|\widehat{J}_1(t,\xi)-\widehat{G}_1(t,|\xi|)\widehat{\Psi}_{1,2}|&\lesssim\chi_{\intt}(\xi)\mathrm{e}^{-c|\xi|^2t}(|\widehat{\psi}_1|+|\widehat{\psi}_2|),\label{Est-03}
\end{align}
and
\begin{align}
\chi_{\intt}(\xi)|\widehat{J}_0(t,\xi)-\widehat{G}_0(t,|\xi|)\widehat{\psi}_2-\widehat{H}_0(t,|\xi|)\widehat{\psi}_2|&\lesssim \chi_{\intt}(\xi)\mathrm{e}^{-c|\xi|^2t}|\widehat{\psi}_2|,\label{Est-02}
\end{align}
with the combined data $\widehat{\Psi}_{1,2}:=\widehat{\psi}_1+\frac{2\kappa-\delta}{2}\widehat{\psi}_2$.
\end{prop}
\begin{proof} Let us initially recall $P_0$ in \eqref{P0}. Thus, we can do the next decomposition:
	\begin{align*}
	\widehat{J}_0(t,\xi)-\widehat{G}_0(t,|\xi|)\widehat{\psi}_2&=\frac{P_0-|\xi|^2}{P_0|\xi|^2}\cos(\lambda_{\mathrm{I}}t)\mathrm{e}^{\lambda_{\mathrm{R}}t}\widehat{\psi}_2-\frac{1}{|\xi|^2}\big(\cos(\lambda_{\mathrm{I}}t)-\cos(|\xi|t)\big)\mathrm{e}^{\lambda_{\mathrm{R}}t}\widehat{\psi}_2\\
	&\quad-\frac{\cos(|\xi|t)}{|\xi|^2}\left(\mathrm{e}^{\lambda_{\mathrm{R}}t}-\mathrm{e}^{-\frac{\delta}{2}|\xi|^2t}\right)\widehat{\psi}_2+\frac{|\xi|^2-P_0}{P_0|\xi|^2}\mathrm{e}^{\lambda_1t}\widehat{\psi}_2+\frac{1}{|\xi|^2}\left(\mathrm{e}^{\lambda_1t}-\mathrm{e}^{-\kappa|\xi|^2t}\right)\widehat{\psi}_2\\
	&=:I_1+I_2+I_3+I_4+I_5.
	\end{align*}
Because of $P_0-|\xi|^2=(\lambda_{\mathrm{R}}-\lambda_1)^2+(\lambda_{\mathrm{I}}-|\xi|)(\lambda_{\mathrm{I}}+|\xi|)=\ml{O}(|\xi|^4)$ for $\xi\in\ml{Z}_{\intt}(\varepsilon_0)$ and $|\cos(\lambda_{\mathrm{I}}t)|\leqslant 1$, one can derive
\begin{align*}
\chi_{\intt}(\xi)(|I_1|+|I_4|)\lesssim\chi_{\intt}(\xi)\mathrm{e}^{-c|\xi|^2t}|\widehat{\psi}_2|.
\end{align*}
Using the integration form to represent the difference
\begin{align*}
\chi_{\intt}(\xi)\left|\mathrm{e}^{\lambda_{\mathrm{R}}t}-\mathrm{e}^{-\frac{\delta}{2}|\xi|^2t}\right|&\lesssim\chi_{\intt}(\xi)\mathrm{e}^{-\frac{\delta}{2}|\xi|^2t}|\xi|^4t\int_0^1\mathrm{e}^{-\frac{\hat{\delta}}{2}|\xi|^4t\tau+\ml{O}(|\xi|^5)t\tau}\mathrm{d}\tau\notag\\
&\lesssim\chi_{\intt}(\xi)|\xi|^4t\mathrm{e}^{-c|\xi|^2t},
\end{align*}
we deduce immediately
\begin{align*}
\chi_{\intt}(\xi)(|I_3|+|I_5|)\lesssim\chi_{\intt}(\xi)|\xi|^2t\mathrm{e}^{-c|\xi|^2t}|\widehat{\psi}_2|\lesssim\chi_{\intt}(\xi)\mathrm{e}^{-c|\xi|^2t}|\widehat{\psi}_2|.
\end{align*}
Indeed, the worse part occurs in $I_2$ since the oscillation cannot completely compensate the singularity as $|\xi|\to0$ under this situation. According to the differential mean value theorem, there exists $\eta_1$ between $|\xi|$ and $\lambda_{\mathrm{I}}$ such that
\begin{align}\label{Bad-mean}
\cos(\lambda_{\mathrm{I}}t)-\cos(|\xi|t)=-\left(\frac{4\hat{\delta}-\kappa\delta^2}{8\kappa}|\xi|^3t+\ml{O}(|\xi|^5)t\right)\sin(\eta_1t).
\end{align}
Associating with the last equality and $|\sin(\eta_1t)|\leqslant 1$, it yields
\begin{align*}
\chi_{\intt}(\xi)|I_2|\lesssim\chi_{\intt}(\xi)|\xi|t\mathrm{e}^{-c|\xi|^2t}|\widehat{\psi}_2|\lesssim \chi_{\intt}(\xi)t^{\frac{1}{2}}\mathrm{e}^{-c|\xi|^2t}|\widehat{\psi}_2|.
\end{align*}
Therefore, we may claim \eqref{Est-01} by summarizing the previous estimates. By the same way as the above, another estimate \eqref{Est-03} can be proved.

To derive faster decay estimates when one subtracts the second-order profile, we observe from \eqref{Bad-mean} that
\begin{align*}
\cos(\lambda_{\mathrm{I}}t)-\cos(|\xi|t)+\frac{4\hat{\delta}-\kappa\delta^2}{8\kappa}|\xi|^3t\sin(|\xi|t)=t^2\ml{O}(|\xi|^6),
\end{align*}
where we used Taylor's expansion. Consequently, the next estimate holds:
\begin{align*}
&\chi_{\intt}(\xi)\left| I_2-\widehat{H}_0(t,|\xi|)\widehat{\psi}_2\right|\\
&\qquad\lesssim\chi_{\intt}(\xi)\left|I_2-\frac{4\hat{\delta}-\kappa\delta^2}{8\kappa}|\xi|t\sin(|\xi|t)\mathrm{e}^{\lambda_{\mathrm{R}}t}\widehat{\psi}_2\right|+\chi_{\intt}(\xi)\left|\frac{4\hat{\delta}-\kappa\delta^2}{8\kappa}|\xi|t\sin(|\xi|t)\mathrm{e}^{\lambda_{\mathrm{R}}t}-\widehat{H}_0(t,|\xi|) \right||\widehat{\psi}_2|\\
&\qquad\lesssim\chi_{\intt}(\xi)|\xi|^4(1+|\xi|)t^2\mathrm{e}^{-c|\xi|^2t}|\widehat{\psi}_2|\\
&\qquad\lesssim\chi_{\intt}(\xi)\mathrm{e}^{-c|\xi|^2t}|\widehat{\psi}_2|,
\end{align*}
which completes the proof of \eqref{Est-02}. 
\end{proof}

\section{Asymptotic behaviors of solution}\label{Section-4}
The schedule of this section is arranged by: in Subsection \ref{Subsec-optimal-solution}, we will prove the optimal estimate \eqref{Opt-Est} of the solution in the $L^2$ norm; in Subsection \ref{Subsec-optimal-profile}, we will demonstrate the optimal estimates \eqref{Opt-Prof}-\eqref{eq:12} of the solution subtracting the first-order profile, i.e. the first-order error term, where last data is taken from energy space with additional $L^{1,1}$ regularity. As a by-product, some estimates for the second-order profile of solution \eqref{Upper-second} will be derived in Subsection \ref{Sub-section-second-order}.

Before implement our schedule, let us recall three optimal estimates (that is, same behaviors of the upper and lower bounds) in the next lemma, whose proofs are based on WKB method as well as refined time-frequencies analysis. 
\begin{lemma}\label{Lemma-Optimal-Est}
Let $n\geqslant 1$. Then, the following optimal estimates for the multipliers hold:
\begin{align*}
\left\|\chi_{\intt}(\xi)\mathrm{e}^{-c|\xi|^2t}\right\|_{L^2}&\asymp t^{-\frac{n}{4}},\\
 \left\|\chi_{\intt}(\xi)|\xi|^{-1}|\sin(|\xi|t)|\mathrm{e}^{-c|\xi|^2t}\right\|_{L^2}&\asymp \ml{D}_n^{(1)}(t),\\
 \left\|\chi_{\intt}(\xi)|\xi|^{-2}|\sin(|\xi|t)|^2\mathrm{e}^{-c|\xi|^2t}\right\|_{L^2}&\asymp \ml{D}_n^{(2)}(t),
\end{align*}
with some  constants $c>0$ for any $t\gg1$, where the time-dependent functions $\ml{D}_n^{(1)}(t)$ and $\ml{D}_n^{(2)}(t)$ were introduced in Subsection \ref{Sub-sec-notation}.
\end{lemma}
Concerning Lemma \ref{Lemma-Optimal-Est}, the first two estimates have been proved by \cite{Ikehata=2014,Ikehata-Onodera=2017}. The upper bound estimate in the third one has been derived by \cite{Chen-Ikehata=2021}, and the sharp lower bound estimate has been obtained by \cite{Chen-Ikehata-Palmieri=2023}. We emphasize that the stronger singularity $|\xi|^{-2}$ as $|\xi|\to0$ appears in the third estimate so that the proof of sharp lower bound is delicate. Two direct consequences from Lemma \ref{Lemma-Optimal-Est} and the proof of Proposition \ref{Prop-Zero-Order} are
\begin{align}\label{Sup_01}
	\|\chi_{\intt}(\xi)|\xi|\widehat{G}_0(t,|\xi|)\|_{L^2}\asymp\ml{D}_n^{(1)}(t)\ \  \mbox{as well as} \ \ \|\chi_{\intt}(\xi)\widehat{G}_0(t,|\xi|)\|_{L^2}\asymp\ml{D}_n^{(2)}(t) 
\end{align}
for $t\gg1$ and any $n\geqslant 1$. In the above statements, we have treated the singular multiplier in the next way:
\begin{align*}
	|\xi|\widehat{G}_0(t,|\xi|)&=\frac{1}{|\xi|}\cos(|\xi|t)\left(\mathrm{e}^{-\kappa|\xi|^2t}-\mathrm{e}^{-\frac{\delta}{2}|\xi|^2t}\right)+\frac{1}{|\xi|}\big(1-\cos(|\xi|t)\big)\mathrm{e}^{-\kappa|\xi|^2t}\\
	&=\frac{\delta-2\kappa}{2}|\xi|t\cos(|\xi|t)\mathrm{e}^{-\frac{\delta}{2}|\xi|^2t}\int_0^1\mathrm{e}^{\frac{\delta-2\kappa}{2}|\xi|^2t\tau}\mathrm{d}\tau+\frac{2\sin^2(\frac{|\xi|}{2}t)}{|\xi|}\mathrm{e}^{-\kappa|\xi|^2t}.
\end{align*}

\subsection{Optimal estimate for the solution}\label{Subsec-optimal-solution}
Applying Proposition \ref{Prop-Zero-Order} into the solution formula \eqref{Representation_Fourier_u}, the next estimate holds:
\begin{align*}
	\chi_{\intt}(\xi)|\widehat{\psi}(t,\xi)|&\lesssim\chi_{\intt}(\xi)\mathrm{e}^{-c|\xi|^2t}|\widehat{\psi}_0|+\chi_{\intt}(\xi)\left(1+\frac{|\sin(|\xi|t)|}{|\xi|}\right)\mathrm{e}^{-c|\xi|^2t}|\widehat{\psi}_1|\\
	&\quad+\chi_{\intt}(\xi)\left(t+\frac{|\sin(|\xi|t)|}{|\xi|}+\frac{|\sin(\frac{|\xi|}{2}t)|^2}{|\xi|^2}\right)\mathrm{e}^{-c|\xi|^2t}|\widehat{\psi}_2|.
\end{align*}
With the help of Lemma \ref{Lemma-Optimal-Est}, we may derive
\begin{align*}
\|\chi_{\intt}(\xi)\widehat{\psi}(t,\xi)\|_{L^2}\lesssim\ml{D}_n^{(1)}(t)\sum\limits_{j=0,1}\|\psi_j\|_{L^{1}}+\ml{D}_n^{(2)}(t)\|\psi_2\|_{L^{1}},
\end{align*}
in which $\ml{D}_n^{(2)}(t)$ plays a dominant role on the right-hand side for $t\gg1$. Concerning bounded and large frequencies from \eqref{Estimate-large-middle}, an exponential decay estimate holds
\begin{align*}
\left\|\big(1-\chi_{\intt}(\xi)\big)\widehat{\psi}(t,\xi)\right\|_{L^2}\lesssim\mathrm{e}^{-ct}\sum\limits_{j=0,1,2}\|\psi_j\|_{H^{-2j}}.
\end{align*}
As a consequence, the Plancherel theorem associated with the last two derived estimates shows the upper bound estimate of \eqref{Opt-Est} for all $n\geqslant 1$.

In order to arrive at the sharp lower bound estimates, we propose the next bridge.
\begin{prop}\label{Prop-4.1}
Let $n\geqslant 1$. Then, the solution to the Cauchy problem \eqref{Eq-Blackstock} fulfills the following refined estimate:
\begin{align*}
\left\| \chi_{\intt}(D)\big(\psi(t,\cdot)-G_0(t,|D|)\psi_2\big) \right\|_{L^2}\lesssim t^{-\frac{n}{4}} \|\psi_0 \|_{L^{1}}+\ml{D}_n^{(1)}(t)\sum\limits_{j=1,2}\|\psi_j \|_{L^{1}},
\end{align*}
and the further one:
\begin{align*}
\left\|\chi_{\intt}(D)\big(\psi(t,\cdot)-G_0(t,|D|)\psi_2-G_1(t,|D|)\Psi_{1,2}-H_0(t,|D|)\psi_2\big)\right\|_{L^2}\lesssim t^{-\frac{n}{4}}\sum\limits_{j=0,1,2} \|\psi_j \|_{L^{1}},
\end{align*}
for $t\gg1$ with the combined data $\Psi_{1,2}:=\psi_1+\frac{2\kappa-\delta}{2}\psi_2$, where the symbols of differential operators $G_j(t,|D|)$ and $H_0(t,|D|)$ were defined in \eqref{Sym-01} and \eqref{Sym-02}, respectively.
\end{prop}
\begin{proof}
The application of triangle inequality for Propositions \ref{Prop-3.3} as well as \ref{Prop-3.4} implies
\begin{align*}
\chi_{\intt}(\xi)|\widehat{\psi}(t,\xi)-\widehat{G}_0(t,|\xi|)\widehat{\psi}_2|&\lesssim\chi_{\intt}(\xi)|\widehat{\psi}(t,\xi)-\widehat{J}_0(t,\xi)|+\chi_{\intt}(\xi)|\widehat{J}_0(t,\xi)-\widehat{G}_0(t,|\xi|)\widehat{\psi}_2|\\
&\lesssim\chi_{\intt}(\xi)\mathrm{e}^{-c|\xi|^2t}\left(|\widehat{\psi}_0|+\frac{|\sin(|\xi|t)|}{|\xi|}|\widehat{\psi}_1|+\left(t^{\frac{1}{2}}+\frac{|\sin(|\xi|t)|}{|\xi|}\right)|\widehat{\psi}_2|\right).
\end{align*}
Hence, for $t\gg1$ we claim
\begin{align*}
\left\|\chi_{\intt}(\xi) \big( \widehat{\psi}(t,\xi)-\widehat{G}_0(t,|\xi|)\widehat{\psi}_2 \big)\right\|_{L^2}\lesssim t^{-\frac{n}{4}} \|\psi_0 \|_{L^{1}}+\ml{D}_n^{(1)}(t)\sum\limits_{j=1,2} \|\psi_j \|_{L^{1}},
\end{align*}
where we employed Lemma \ref{Lemma-Optimal-Est}. Following the same procedure as the proof of \eqref{Opt-Est} from the upper bound's viewpoint, we complete the first aim. For another, we also can get
\begin{align*}
	\chi_{\intt}(\xi)|\widehat{\psi}(t,\xi)-\widehat{G}_0(t,|\xi|)\widehat{\psi}_2-\widehat{G}_1(t,|\xi|)\widehat{\Psi}_{1,2}-\widehat{H}_0(t,|\xi|)\widehat{\psi}_2|\lesssim\chi_{\intt}(\xi)\mathrm{e}^{-c|\xi|^2t}(|\widehat{\psi}_0|+|\widehat{\psi}_1|+|\widehat{\psi}_2|).
\end{align*}
Then, repeating the same approach as before, we conclude our second desired estimate.
\end{proof}

Later, we will use the following chain several times:
\begin{align*}
	\chi_{\extt}(\xi)|\xi|^{k_0}\mathrm{e}^{-c|\xi|^2t}&=\chi_{\extt}(\xi)|\xi|^{-k_1}(|\xi|^2t)^{\frac{k_0+k_1}{2}}t^{-\frac{k_0+k_1}{2}}\mathrm{e}^{-c|\xi|^2t}\\
	&\lesssim\chi_{\extt}(\xi)|\xi|^{-k_1}t^{-\frac{k_0+k_1}{2}}\mathrm{e}^{-2ct}\\
	&\lesssim\chi_{\extt}(\xi)|\xi|^{-k_1}\mathrm{e}^{-ct}
\end{align*}
for $t\gg1$, where $k_0,k_1$ are positive constants.
Taking the notations
\begin{align*}
	F_1(t,x)&:=\psi(t,x)-G_0(t,|D|)\psi_2, \\
	F_2(t,x)&:=G_0(t,|D|)\psi_2-G_0(t,x)P_{\psi_2},
\end{align*}
we have 
\begin{align*}
	\|F_1(t,\cdot)\|_{L^2} & \lesssim \left\|\chi_{\intt}(\xi) \big( \widehat{\psi}(t,\xi)-\widehat{G}_0(t,|\xi|)\widehat{\psi}_2 \big) \right\|_{L^2} + \left\|\big(1-\chi_{\intt}(\xi)\big) \widehat{\psi}(t,\xi) \right\|_{L^2}\\
	&\quad+ \left\|\big(1-\chi_{\intt}(\xi)\big) \widehat{G}_{0}(t,|\xi|)\widehat{\psi}_2 \right\|_{L^2}\\
	& \lesssim \ml{D}_n^{(1)}(t)\sum\limits_{j=0,1,2} \|\psi_j \|_{ H^{-2j}\cap L^{1}}.
\end{align*}
On the other hand, we able able to also obtain 
\begin{align*}
	\|F_2(t,\cdot)\|_{L^2} =o\big(\ml{D}_n^{(2)}(t)\big)
\end{align*}
as $t \to \infty$, just following the estimate of $E_2(t,x)$ below (see the derivation of \eqref{eq:26} indeed).
As a consequence, we conclude the desired estimate \eqref{eq:6} by 
\begin{align*}
	\|\psi(t,\cdot)-\psi^{(1)}(t,\cdot)\|_{L^2} \lesssim \|F_1(t,\cdot)\|_{L^2} +\|F_2(t,\cdot)\|_{L^2}  = o\big( \ml{D}_n^{(2)}(t)\big)
\end{align*} 
for large $t \gg 1$.

Eventually, since \eqref{Sup_01} and \eqref{eq:6} for large-time, Minkowski's inequality results
\begin{align*}
	\|\psi(t,\cdot)\|_{L^2}&\geqslant -\|\psi(t,\cdot)- \psi^{(1)}(t,\cdot) \|_{L^2} + \|\psi^{(1)}(t,\cdot) \|_{L^2} \\
	&\gtrsim\ml{D}_n^{(2)}(t)|P_{\psi_2}|-o\big(\ml{D}_n^{(2)}(t)\big) \\
	&\gtrsim\ml{D}_n^{(2)}(t)|P_{\psi_2}|
\end{align*}
for $t\gg1$ due to $|P_{\psi_2}|\not\equiv0$ and $\psi_j\in L^{1}$ for $j=0,1,2$. It gives immediately our desired result.
\subsection{Error estimates for second-order profile of the solution}\label{Sub-section-second-order}
Our purpose in this part is to prove \eqref{Upper-second} scrupulously. The  insightful decomposition manifests
\begin{align*}
\psi(t,x)-\psi^{(1)}(t,x)-\psi^{(2)}(t,x)=\sum\limits_{j=1,\dots,4}E_j(t,x),
\end{align*}
where all components are determined as follows:
\begin{align*}
	E_1(t,x)&:=\psi(t,x)-G_0(t,|D|)\psi_2-H_0(t,|D|)\psi_2-G_1(t,|D|)\Psi_{1,2},\\
	E_2(t,x)&:=G_0(t,|D|)\psi_2-G_0(t,x)P_{\psi_2}-\nabla G_0(t,x)\circ M_{\psi_2},\\
	E_3(t,x)&:=H_0(t,|D|)\psi_2-H_0(t,x)P_{\psi_2},\\
	E_4(t,x)&:=G_1(t,|D|)\Psi_{1,2}-G_1(t,x)P_{\Psi_{1,2}}.
\end{align*}
The first one can be easily estimated according to Proposition \ref{Prop-4.1} such that
\begin{align*}
	\|\chi_{\intt}(D)E_1(t,\cdot)\|_{L^2}\lesssim t^{-\frac{n}{4}}\sum\limits_{j=0,1,2} \|\psi_j\|_{L^{1}}.
\end{align*}
It is also easily seen that 
\begin{align*}
	\left\|\big(1-\chi_{\intt}(D)\big) E_1(t,\cdot)\right\|_{L^2}\lesssim \mathrm{e}^{-c t}\sum\limits_{j=0,1,2} \|\psi_j\|_{H^{-2j}}
\end{align*}
for $t \gg 1$.

At this moment, we may realize that the worst component is $E_2(t,x)$ exactly, which can be deeply separated by two parts as follows:
\begin{align*}
	E_{2,1}(t,x)&:=\int_{|y|\leqslant t^{\alpha}}\big(G_0(t,x-y)-G_0(t,x)-y\circ\nabla G_0(t,x)\big)\psi_2(y)\mathrm{d}y,\\
	E_{2,2}(t,x)&:=\int_{|y|\geqslant t^{\alpha}}\big(G_0(t,x-y)-G_0(t,x)\big)\psi_2(y)\mathrm{d}y-\int_{|y|\geqslant t^{\alpha}}y\circ\nabla G_0(t,x)\psi_2(y)\mathrm{d}y,
\end{align*}
with a sufficiently small constant $\alpha>0$. Indeed, Taylor's expansions give
\begin{align*}
|G_0(t,x-y)-G_0(t,x)|&\lesssim|y|\,|\nabla G_0(t,x-\theta_0 y)|,\\
|G_0(t,x-y)-G_0(t,x)-y\circ\nabla G_0(t,x)|&\lesssim|y|^2\,|\nabla^2 G_0(t,x-\theta_1 y)|,
\end{align*}
equipping some constants $\theta_0,\theta_1\in(0,1)$. For one thing, we know
\begin{align*}
\|E_{2,1}(t,\cdot)\|_{L^2}&\lesssim t^{2\alpha}\left\||\xi|^2\widehat{G}_0(t,|\xi|)\right\|_{L^2} \| \psi_2 \|_{L^{1}}\lesssim t^{2\alpha-\frac{n}{4}} \| \psi_2 \|_{L^{1}}.
\end{align*}
For another, since $\psi_2\in L^{1,1}$, it holds
\begin{align*}
	\lim\limits_{t\to\infty}\int_{|y|\geqslant t^{\alpha}}|y|\,|\psi_2(y)|\mathrm{d}y=0.
\end{align*}
In other words, 
we have
\begin{align*}
	\|E_{2,2}(t,\cdot)\|_{L^2}&\lesssim\||\xi|\widehat{G}_0(t,|\xi|)\|_{L^2}\int_{|y|\geqslant t^{\alpha}}|y|\,|\psi_2(y)|\mathrm{d}y =o\big(\ml{D}_n^{(1)}(t) \big)
\end{align*}
for $t\gg1$. By choosing $\alpha>0$ to be a small number, we summarize the obtained estimates to assert
\begin{align} \label{eq:26}
	\|E_2(t,\cdot)\|_{L^2}  =o\big(\ml{D}_n^{(1)}(t) \big)
\end{align}
for any $n\geqslant 1$ as well as $t\gg1$.
With the same procedure as \eqref{eq:26}, it holds
\begin{align*}
	\|E_3(t,\cdot)\|_{L^2}+\|E_4(t,\cdot)\|_{L^2}=o\big(\ml{D}_n^{(1)}(t) \big) 
\end{align*}
for $t\gg1$, under the assumption on $\psi_j\in H^{-2j}\cap L^{1}$ for $j=1,2$.
%
%
Finally, summing up the above estimates, we now can complete the proof of \eqref{Upper-second}.

\subsection{Optimal estimate for the leading term}\label{Subsec-optimal-profile}
We begin with the estimate for
\begin{align*}
	\|\psi(t,\cdot)-\psi^{(1)}(t,\cdot)\|_{L^2}
\end{align*}
 under additional $L^{1,1}$ regularity data. Denoting 
\begin{align*}
A_{\psi_2}(\xi):=\int_{\mb{R}^n}\psi_2(x)\big(\cos(x\cdot\xi)-1\big)\mathrm{d}x\ \ \mbox{and}\ \ B_{\psi_2}(\xi):=\int_{\mb{R}^n}\psi_2(x)\sin(x\cdot\xi)\mathrm{d}x,
\end{align*}
so that $\widehat{\psi}_2-P_{\psi_2}=A_{\psi_2}(\xi)-i B_{\psi_2}(\xi)$, one claims
\begin{align}\label{Est-04}
\left\|\chi_{\intt}(D)\big(G_0(t,|D|)\psi_2-G_0(t,\cdot)P_{\psi_2}\big)\right\|_{L^2}&\lesssim\left\|\chi_{\intt}(\xi)\widehat{G}_0(t,|\xi|)\big(A_{\psi_2}(\xi)-iB_{\psi_2}(\xi)\big)\right\|_{L^2}\notag\\
&\lesssim\|\chi_{\intt}(\xi)|\xi|\widehat{G}_0(t,|\xi|)\|_{L^2}\|\psi_2\|_{L^{1,1}}\notag\\
&\lesssim\ml{D}_n^{(1)}(t)\|\psi_2\|_{L^{1,1}}
\end{align}
for $t\gg1$, where we used \cite[Lemma 2.2]{Ikehata=2014} indicating
\begin{align*}
	|A_{\psi_2}(\xi)|+	|B_{\psi_2}(\xi)|\lesssim|\xi|\,\|\psi_2\|_{L^{1,1}}.
\end{align*}
The combination of Proposition \ref{Prop-4.1} as well as \eqref{Est-04} immediately shows
\begin{align*}
\left\| \chi_{\intt}(D)\big(\psi(t,\cdot)-G_0(t,\cdot)P_{\psi_2}\big) \right\|_{L^2}\lesssim\ml{D}_n^{(1)}(t) \left( \sum\limits_{j=0,1}\|\psi_j\|_{L^{1}} + \|\psi_2 \|_{L^{1,1}} \right).
\end{align*}
Associated with the next control:
\begin{align*}
\left\| \big(1-\chi_{\intt}(D)\big)\big(\psi(t,\cdot)-G_0(t,\cdot)P_{\psi_2}\big) \right\|_{L^2}\lesssim\mathrm{e}^{-ct}\sum\limits_{j=0,1,2}\|\psi_j\|_{H^{-2j}}+\mathrm{e}^{-ct}|P_{\psi_2}|,
\end{align*}
we are able to derive the estimate \eqref{Opt-Prof} from the above easily.

Nevertheless, to derive the optimal lower bound estimate, we would like to estimate the profile $\psi^{(2)}(t,\cdot)$ in the $L^2$ norm from the below carefully because
\begin{align}\label{eq.sum}
\|\psi(t,\cdot)-\psi^{(1)}(t,\cdot)\|_{L^2}\geqslant\underbrace{\|\psi^{(2)}(t,\cdot)\|_{L^2}}_{\mbox{next target}}-\underbrace{\|\psi(t,\cdot)-\psi^{(1)}(t,\cdot)-\psi^{(2)}(t,\cdot)\|_{L^2}}_{=o(\ml{D}_n^{(1)}(t))}.
\end{align}
Reminding the definitions of $P_{\Psi_{1,2}}$ and $P_{2}$, 
we have the expression of $\widehat{\psi}^{(2)}(t,\xi)$ as follows:
\begin{align*}
\widehat{\psi}^{(2)}(t,\xi)=i (\xi \circ M_{\psi_2} ) \widehat{G}_0(t,|\xi|)+ \left( P_{\Psi_{1,2}} +P_{2} t |\xi|^{2} \right) \widehat{G}_1(t,|\xi|).
\end{align*}
This gives 
\begin{align*}
|\widehat{\psi}^{(2)}(t,\xi)|^{2}=(\xi \circ M_{\psi_2} )^{2} |\widehat{G}_0(t,|\xi|)|^{2}+ \left( P_{\Psi_{1,2}} +P_{2} t |\xi|^{2} \right)^{2} |\widehat{G}_1(t,|\xi|)|^{2},
\end{align*}
and integrating the resultant over $\mb{R}^n$,  one deduces
\begin{align*}
\| \widehat{\psi}^{(2)}(t,\xi)\|_{L^{2}}^{2}&=\int_{\mathbb{R}^{n}}
(\xi \circ M_{\psi_2} )^{2} |\widehat{G}_0(t,|\xi|)|^{2} \mathrm{d} \xi+\int_{\mathbb{R}^{n}}
\left( P_{\Psi_{1,2}} +P_{2} t |\xi|^{2} \right)^{2} |\widehat{G}_1(t,|\xi|)|^{2}
\mathrm{d} \xi\\
&=:A_{1}(t) +A_{2}(t).
\end{align*}
We now estimate $A_{1}(t)$ and $A_{2}(t)$ separately. 
At first, for the estimate of $A_{1}(t)$, noting the observations that 
$ |\xi|^2|\widehat{G}_0(t,|\xi|)|^{2}$ is radial symmetric with respect to $\xi$ and 
\begin{align*}
\int_{\mathbb{S}^{n-1}} \omega_{j} \omega_{k} \mathrm{d} \omega =0
\end{align*}
for $1 \leqslant j<k \leqslant n$, 
we arrive at 
\begin{align*}
A_{1}(t) & \gtrsim \int_{0}^{\infty} r^{2} | \widehat{G}_0(t,r)|^{2} r^{n-1} \mathrm{d}r \int_{\mathbb{S}^{n-1}} (\omega \circ M_{\psi_2} )^{2} \mathrm{d} \omega  \\
&\gtrsim \|\chi_{\intt}(\xi)|\xi|\widehat{G}_0(t,|\xi|)\|_{L^2}^2|M_{\psi_2}|^2\\
&\gtrsim\big(\ml{D}_n^{(1)}(t)\big)^{2} |M_{\psi_2}|^{2}, 
\end{align*}
which implies the estimate 
\begin{align}\label{eq:31}
	\| \psi^{(2)}(t,\cdot)\|_{L^{2}}=\| \widehat{\psi}^{(2)}(t,\xi)\|_{L^{2}} \gtrsim \ml{D}_n^{(1)}(t) |M_{\psi_2}|. 
\end{align}

In what follows in the passage, we estimate $A_{2}(t)$.
To do so, we decompose  $A_{2}(t)$ into three parts: 
\begin{align*}
A_{2}(t)=A_{2,1}(t)+A_{2,2}(t)+A_{2,3}(t),
\end{align*}
carrying 
\begin{align*}
A_{2,1}(t) & := P_{\Psi_{1,2}}^{2} \int_{\mathbb{R}^{n}} |\widehat{G}_1(t,|\xi|)|^{2} \mathrm{d} \xi= P_{\Psi_{1,2}}^{2} t^{1-\frac{n}{2}} |\mathbb{S}^{n-1}| \int_{0}^{\infty} \mathrm{e}^{-\delta r^{2}} \sin^{2}(\sqrt{t} r)r^{n-3} \mathrm{d}r,  \\ 
A_{2,2}(t) & := 2P_{\Psi_{1,2}} P_{2} t^{1-\frac{n}{2}} |\mathbb{S}^{n-1}| \int_{0}^{\infty} \mathrm{e}^{-\delta r^{2}} \sin^{2}(\sqrt{t} r)r^{n-1} \mathrm{d}r,\\
A_{2,3}(t) & := P_{2}^2 t^{1-\frac{n}{2}} |\mathbb{S}^{n-1}| \int_{0}^{\infty} \mathrm{e}^{-\delta r^{2}} \sin^{2}(\sqrt{t} r)r^{n+1} \mathrm{d}r,
\end{align*}
where $|\mathbb{S}^{n-1}|$ represents the surface measure of $\mathbb{S}^{n-1}$.

When $n=1,2$, applying the argument in \cite{Ikehata-Onodera=2017}, we see that 
\begin{align*}
A_{2,1}(t) & \geqslant P_{\Psi_{1,2}}^{2} \big(\ml{D}_n^{(1)}(t)\big)^{2}, \\  
|A_{2,2}(t)|+|A_{2,3}(t)| & \lesssim (|P_{\Psi_{1,2}}P_2|+P_2^2) t^{1-\frac{n}{2}}.
\end{align*}
Therefore, we obtain 
\begin{align} \label{eq:32}
A_{2}(t) \geqslant A_{2,1}(t)- |A_{2,2}(t)|- |A_{2,3}(t)|  \gtrsim P_{\Psi_{1,2}}^{2} \big(\ml{D}_n^{(1)}(t)\big)^{2}
\end{align}
for large $t$.
Combining \eqref{eq:31} and \eqref{eq:32}, we arrive at the desired estimate 
\begin{align*}
\|\psi^{(2)}(t,\cdot)\|_{L^2}\gtrsim \left(|M_{\psi_2}|+|P_{\Psi_{1,2}}|\right)\ml{D}_n^{(1)}(t)
\end{align*}
if $n=1,2$ for $t\gg1$.

On the other hand, when $n \geqslant 3$
motivated by \cite{Ikehata=2014} with the facts that $\mathrm{e}^{-\delta r^2}r^{n-3} \in L^{1}(0,\infty)$ and $2\sin^{2}(\sqrt{t} r) =1-\cos(2 \sqrt{t} r)$,
we have
\begin{align}\label{eq:33}
\int_{0}^{\infty} \mathrm{e}^{-\delta r^{2}} \sin^{2}(\sqrt{t} r)r^{n-3} \mathrm{d}r & =\frac{1}{2} \int_{0}^{\infty} \mathrm{e}^{-\delta r^{2}} r^{n-3} \mathrm{d}r-\frac{1}{2} \int_{0}^{\infty} \mathrm{e}^{-\delta r^{2}} \cos(2 \sqrt{t} r)r^{n-3} \mathrm{d}r\notag \\
& = \frac{1}{2} \delta^{-\frac{n}{2}+1} \int_{0}^{\infty} \mathrm{e}^{-\tilde{r}^{2}} \tilde{r}^{n-3} \mathrm{d}\tilde{r} +o(1) \notag\\
& = \frac{1}{4} \delta^{-\frac{n}{2}+1} \Gamma \left(\frac{n}{2}-1 \right)  +o(1) 
\end{align}
as $t \to \infty$, where 
\begin{align*}
\Gamma(z):= \displaystyle\int_{0}^{\infty} \mathrm{e}^{-\eta} \eta^{z-1} \mathrm{d}\eta= 2 \int_{0}^{\infty} \mathrm{e}^{-\eta^{2}} \eta^{2z-1} \mathrm{d}\eta
\end{align*}
for $z>0$. 
In the last lines, we applied the change of integral variable $\delta^{\frac{1}{2}} r = \tilde{r}$ and the Riemann-Lebesgue formula.
Similarly, we also investigate 
\begin{align}
\label{eq:34}
\int_{0}^{\infty} \mathrm{e}^{-\delta r^{2}} \sin^{2}(\sqrt{t} r)r^{n-1} \mathrm{d}r 
& = \frac{1}{4} \delta^{-\frac{n}{2}} \Gamma \left(\frac{n}{2} \right)  +o(1),\\
\int_{0}^{\infty} \mathrm{e}^{-\delta r^{2}} \sin^{2}(\sqrt{t} r)r^{n+1} \mathrm{d}r 
& = \frac{1}{4} \delta^{-\frac{n}{2}-1} \Gamma \left(\frac{n}{2}+1 \right)  +o(1),\label{eq:341}
\end{align}
as $t \to \infty$.
Hence, concerning the large-time situation, we recall the well-known relation that
\begin{align*}
\Gamma \left(\frac{n}{2} \right)= \frac{n}{2}\Gamma \left(\frac{n}{2}-1 \right)
\end{align*}
 to see that 
\begin{align} \label{eq:35}
A_{2}(t)\geqslant \frac{1}{4} \delta^{-\frac{n}{2}-1} \Gamma \left(\frac{n}{2}-1 \right) t^{1-\frac{n}{2}} |\mathbb{S}^{n-1}| \left(\delta P_{\Psi_{1,2}}+\frac{n}{2}P_2\right)^2.
\end{align}
by employing \eqref{eq:33} and \eqref{eq:341}.
Summarizing \eqref{eq:31} with \eqref{eq:35}, we conclude the estimate
\begin{align*}
\|\psi^{(2)}(t,\cdot)\|_{L^2}\gtrsim\left(|M_{\psi_2}|+|2\delta P_{\Psi_{1,2}}+nP_2|\right)t^{\frac{1}{2}-\frac{n}{4}}
\end{align*}
if $n\geqslant 3$ for $t\gg1$.

Eventually, from the lower bound estimate of $\|\psi^{(2)}(t,\cdot)\|_{L^2}$ and \eqref{eq.sum}, the proof of desired estimates \eqref{Opt-Prof1} and \eqref{eq:12} are finished.

\section{Final remarks}\label{Section-5}
	In recent years, semilinear Cauchy problems for acoustic waves with power-type nonlinearity have caught a lot of attentions, for example, semilinear Kuznetsov's equation (see \cite{D'Ambrosio-Lucente,Dabbicco-Reissig=2014,Dao-Reissig=2019} referring to semilinear viscoelastic damped waves) and semilinear Moore-Gibson-Thompson equations (see \cite{Chen-Palmieri=2020,Chen-Ikehata=2021,Ming-Yang-Fan-Yao=2021} and references therein). If one is interested in the semilinear Blackstock's model with power-type nonlinear term, namely,
	\begin{align}\label{Eq-Semilinear-Blackstock}
		\begin{cases}
			(\partial_t-\kappa\Delta)(\psi_{tt}-\Delta\psi-\delta\Delta\psi_t)+\kappa(\gamma-1)(b\nu-\kappa)\Delta^2\psi_t=|\psi|^p,&x\in\mb{R}^n,\ t>0,\\
			\psi(0,x)=\psi_0(x),\ \psi_t(0,x)=\psi_1(x),\ \psi_{tt}(0,x)=\psi_2(x),&x\in\mb{R}^n,
		\end{cases}
	\end{align}
	with $\kappa>0$, $\delta>0$, $\gamma\in(1,\frac{5}{3}]$ and $b\nu>0$, the next results may be obtained by using derived estimate \eqref{Opt-Est} in this work without technical challenging.
	\begin{itemize}
		\item There  uniquely exists a global (in time) energy solution
		\begin{align*}
			\psi\in\ml{C}^2([0,\infty),H^4)\cap \ml{C}^1([0,\infty),H^2)\cap \ml{C}([0,\infty), L^2)
		\end{align*}
	such that
	\begin{align*}
		\sum\limits_{\ell=0,1,2}\left((1+t)^{-\frac{2-\ell}{2}+\frac{n}{4}}\|\partial_t^{\ell}\psi(t,\cdot)\|_{L^2}+(1+t)^{\frac{2-\ell}{2}+\frac{n}{4}}\|\partial_t^{\ell}\psi(t,\cdot)\|_{\dot{H}^{4-2\ell}}\right)\lesssim\sum\limits_{\ell=0,1,2}\|\psi_{\ell}\|_{H^{4-\ell}\cap L^1},
	\end{align*}
		providing that 
		\begin{itemize}
			\item[$\diamond$] $2\leqslant p\leqslant\frac{n}{\max\{n-8,0\}}$ from applications of Gagliardo-Nirenberg inequality
			\begin{align*}
				\|\,|\psi(t,\cdot)|^p\|_{L^{m}}\lesssim\|\psi(t,\cdot)\|_{L^{mp}}^p\lesssim\|\psi(t,\cdot)\|_{L^2}^{(1-\beta_m)p}\|\psi(t,\cdot)\|_{\dot{H}^4}^{\beta_m p}
			\end{align*}
			with $[0,1]\ni\beta_m:=\frac{n}{4}(\frac{1}{2}-\frac{1}{mp})$ for $m=1,2$;
			\item[$\diamond$] $ p>\frac{n+2}{n-2}$ for $n\geqslant 5$ from the integrable condition
			\begin{align*}
				\int_0^{t/2}(1+t-\tau)^{1-\frac{n}{4}}\|\,|\psi(\tau,\cdot)|^p\|_{L^2\cap L^1}\mathrm{d}\tau\lesssim(1+t)^{1-\frac{n}{4}}\int_0^{t/2}(1+\tau)^{(1-\frac{n}{2})p+\frac{n}{2}}\mathrm{d}\tau<\infty,
			\end{align*}
		\end{itemize}
		 by assuming small datum belonging. Its proof is standard basing on the contraction mapping, e.g. the framework used in \cite{Palmieri-Reissig=2018}.
		\item The global (in time) weak solution does not exist if $1<p<\infty$ for $n=1,2$, and $1<p\leqslant\frac{n+1}{n-2}$ for $n\geqslant 3$, by assuming positive datum (in the integral sense) in $L^1$. This result implies optimality for $n=1,2$. Its proof is based on the test function method, e.g. some generalizations of \cite{Zhang=2001}.
	\end{itemize}
	Whereas, due to the gap $p\in(\frac{n+1}{n-2},\frac{n+2}{n-2}]$ with $n\geqslant 3$ for indeterminacy of global existence or blow-up finite time, the critical exponent $p=p_{\mathrm{crit}}(n)$ for the semilinear model \eqref{Eq-Semilinear-Blackstock} for $n\geqslant 3$ is still open ($p_{\mathrm{crit}}(n)=\infty$ for $n=1,2$), where the critical exponent denotes the threshold for global (in time) existence of solution and blow-up of solution.


\end{document}